\documentclass[12pt]{article}
\usepackage{amsmath,amssymb,amsthm}
\usepackage{amsfonts, amsmath}
\newtheorem{theorem}{Theorem}[section]
\newtheorem{lemma}{Lemma}[section]

\newtheorem{proposition}{Proposition}[section]
\theoremstyle{definition}
\newtheorem{definition}{Definition}[section]
\theoremstyle{remark}

\numberwithin{equation}{section}

 \title{A mathematically derived definitional/semantical theory of truth 
} 
 
\author {S. Heikkil\"a\\
Department of Mathematical Sciences, University of Oulu\\
BOX 3000, FIN-90014, Oulu, Finland\\
E-mail: sheikki@cc.oulu.fi}
\begin{document}
\maketitle 

\noindent
\begin{abstract} 
\noindent
Ordinary and transfinite recursion and induction and ZF set theory are used to construct from a fully interpreted object language and from an  extra formula a new language. It is fully interpreted under a suitably defined interpretation. This interpretation is equivalent to the interpretation by meanings of sentences if the object language is so interpreted. The added formula provides a truth predicate for the constructed language. The so obtained theory of truth satisfies the norms presented  in  Hannes Leitgeb's paper 'What Theories of Truth Should be Like (but Cannot be)'.  
 
\vskip12pt

\vskip12pt


\end{abstract}


 \section{Introduction}\label{S1} 
Theories of truth are presented for languages.
Based on 'Chomsky Definition' (cf. \cite{[C]}) a language is assumed to be a countably infinite set of well-formed sentences, each of finite length, and constructed out of a finite or a countably infinite set of symbols. 
\smallskip

(i) A theory of syntax of a language is formed by symbols, and rules to construct well-formed sentences. Symbols consist of letters, parentheses, commas, dots, constants containing natural numbers, terms containing numerals, and logical symbols $\neg$ (not), $\vee$ (or), $\wedge$ (and), $\rightarrow$ (implies),  $\leftrightarrow$ (if and only if), $\forall$ (for all) and $\exists$ (exist). If $A$ and $B$ are (denote) sentences, so are $\neg A$, $A\vee B$, $A\wedge B$, $A\rightarrow B$ and $A\leftrightarrow B$.  If $P(x)$ is a formula of a language, and $X_P$ is a set of terms, then $P$ is called a predicate with domain $X_P$ if $P(x)$ is a sentence of that language for each assignment of a term of $X_P$ into $x$ (shortly, for each $x\in X_P$). $\forall xP(x)$ and $\exists xP(x)$ are then sentences of the language.	If $P$ has several free variables $x_1,\dots,x_m$, then  $P(x_1,\dots,x_m)$ is denoted by $P(x)$, and the sentences
$\forall xP(x)$ and $\exists xP(x)$ stand for universal and existential closures $\forall\cdots\forall P(x_1,\dots,x_m)$ and $\exists\cdots\exists P(x_1,\dots,x_m)$. Symbols may contain a lexicon of a first-order predicate logic (cf. \cite[Definition II.5.2]{[Ku]}).
\smallskip

An interpretation of  sentences of a language is also needed. 

(ii)  A language is called fully interpreted, if every sentence is interpreted either as true or as false, and if the interpretation of those sentences which contain logical symbols satisfy following rules of classical logic when $A$ and $B$ denote sentences of the language ('iff' means 'if and only if'): 
$A$ is true iff $\neg A$ is false, and $A$ is false iff $\neg A$ is true; $A\vee B$ is true iff $A$ or $B$ is true, and false iff  $A$ and $B$ are false; $A\wedge B$ is true iff  $A$ and $B$ are true, and false iff $A$ or $B$ is false; $A\rightarrow B$ is true iff $A$ is false or $B$ is true, and false iff $A$ is true and $B$ is false; $A\leftrightarrow B$ is true iff  $A$ and $B$ are both true or both false, and false iff $A$ is true and $B$ is false or $A$ is false and $B$ is true. If $P$ is a predicate with domain $X_P$, then 
$\forall xP(x)$ is true iff $P(x)$ is true for every $x\in X_P$, and false iff $P(x)$ is false for some $x\in X_P$;
$\exists xP(x)$ is true iff $P(x)$ is true for some $x\in X_P$, and false iff $P(x)$ is false for every $x\in X_P$.

Any first-order formal language equipped with a consistent theory interpreted by a countable model, and containing natural numbers and numerals, is a fully interpreted language in the above sense. A classical example is the language of arithmetic  with its standard interpretation. Another example is the first-order language of set theory,  
the interpretation being determined by the  minimal model  constructed in \cite{[4]} for ZF set theory. 
\smallskip

'Truth should be expressed by a predicate' is the first requirement presented in \cite{Lei07} for theories of truth.
Many languages, e.g., the above mentioned, don't have such a predicate. Therefore we construct from such a language $L$, from its predicates, and from sentences induced by an additional formula $T(x)$ a new language $\mathcal L$,
and choose a fixed  G\"odel numbering to its sentences.  To each proper subset $U$ of the set $D$ of those G\"odel numbers we  
construct other subsets $G(U)$ and $F(U)$ of $D$.   

In Appendix it is shown that there is the smallest subset $U$ of $D$ which is consistent, i.e., for no  sentence $A$ of $\mathcal L$ the G\"odel numbers of both $A$ and $\neg A$ are in $U$, and  for which $U=G(U)$. The sentences of $\mathcal L$ whose G\"odel numbers are in $G(U)$ or in $F(U)$ and the symbols of $L$ form a language $\mathcal L^0$ that contains $L$.
A sentence of $\mathcal L^0$ is interpreted as true iff its G\"odel number is in $G(U)$, and as false iff its G\"odel number is in $F(U)$. In this interpretation $\mathcal L^0$ is shown to be fully interpreted, if the object language $L$ is fully interpreted. If $L$ is fully interpreted by meanings of its sentences, so is $\mathcal L^0$,
and this interpretation is proved to be equivalent to that defined above. 
Moreover,  the interpretations of $L$ and  $\mathcal L^0$ are
 compatible  in $L$, and $T$ is a truth predicate of $\mathcal L^0$ when its domain $X_T$ is the set of numerals of G\"odel numbers of all sentences of $\mathcal L^0$. This provides a theory of truth for $\mathcal L^0$.  
That theory is shown to satisfy all the norms presented in \cite{Lei07} for truth theories.

Ordinary and transfinite recursions and inductions and ZF set theory for sets of natural numbers are main tools in constructions and proofs.  
As for these tools see, e.g., \cite{[Ku]}.


\section{Recursive constructions}\label{S2} 

Let $L$ be a fully interpreted  language without a truth predicate.  Construct a language $\mathcal L_0$ as follows: Its base language is formed by $L$, an extra formula  $T(x)$ and its assignments when $x$ goes through all numerals.
Fix a G\"odel numbering to the base language. The  G\"odel number of a sentence (denoted by) $A$ is denoted by \#$A$, and the numeral of \#$A$ by $\left\lceil A\right\rceil$.
If $P$ is a predicate of $L$ with domain $X_P$, then $P(x)$ is a sentence of $L$ for each  $x\in X_P$, and $\left\lceil P(x)\right\rceil$ is the numeral of its G\"odel number. Thus  
$T(\left\lceil P(x)\right\rceil)$ is a sentence of $\mathcal L_0$ for each $x\in X_P$, whence  $T(\left\lceil P(\cdot)\right\rceil)$ is a predicate of $\mathcal L_0$ with domain $X_P$.  The construction of $\mathcal L_0$ is completed by adding to it sentences $\forall xT(x)$, $\exists xT(x)$,  
$\forall xT(\left\lceil T(x)\right\rceil)$ and $\exists xT(\left\lceil T(x)\right\rceil)$, and sentences 
$\forall xT(\left\lceil P(x)\right\rceil)$ and $\exists xT(\left\lceil P(x)\right\rceil)$  for every predicate $P$ of $L$. 
\smallskip

When a language $\mathcal L_n$, $n\in\mathbb N_0 =\{0,1,2,\dots\}$, is defined, let $\mathcal L_{n+1}$ be the language which  is formed by the  
 sentences $A$, $\neg A$, $A\vee B$, $A\wedge B$, $A\rightarrow B$ and $A\leftrightarrow B$, where $A$ and $B$ go through all sentences of $\mathcal L_n$. The language $\mathcal L$ is defined as the union of languages $\mathcal L_n$, $n\in\mathbb N_0$. 
Extend the G\"odel numbering of the base language to $\mathcal L$, and denote by $D$ the set of those G\"odel numbers.  
If $A$ is a sentence of $\mathcal L$, denote by '$x=\left\lceil A\right\rceil$' the assignment of the numeral $\left\lceil A\right\rceil$ of the G\"odel number \#$A$ of $A$ into $x$.  

Denote by $\mathcal P$ the set of all predicates of $L$. Divide $\mathcal P$ into three disjoint subsets. 

\begin{equation}\label{E50}
\begin{cases}
\mathcal P_1=\{P\in \mathcal P: P(x) \hbox{ is a true sentence of $L$ for every } x\in X_P \},\\
\mathcal P_2=\{P\in \mathcal P: P(x) \hbox{ is a false sentence of $L$ for every } x\in X_P \},\\
\mathcal P_3=\{P\in \mathcal P: P(x) \hbox{ is a true sentence of $L$ for some but not for all } x\in X_P\}.
\end{cases}
\end{equation}
 
Given a proper subset $U$ of  $D$, construct new subsets $G(U)$ and $F(U)$ of $D$ as follows. Define 
\begin{equation}\label{E20}
\begin{cases}
D_1(U)=\{\hbox{\#$T(x)$: $x=\left\lceil A\right\rceil$, where $A$ is a sentence of $\mathcal L$ and  \#$A$  is in $U$}\},\\
D_2(U)=\{\hbox{\#[$\neg T(x)$]: $x=\left\lceil A\right\rceil$, where $A$ is a sentence of $\mathcal L$ and  \#[$\neg A$]  is in $U$}\},\\
D_1=\{\#[\neg\forall x T(x)],\#[\exists xT(x)],\#[\neg(\forall xT(\left\lceil T(x)\right\rceil))], \#[\exists xT(\left\lceil T(x)\right\rceil)]\},\\
D_2=\{\#[\forall xT(\left\lceil P(x)\right\rceil)], \#[\exists xT(\left\lceil P(x)\right\rceil)]:P\in\mathcal P_1\},\\
D_3=\{\#[\neg(\forall xT(\left\lceil P(x)\right\rceil))], \#[\neg(\exists xT(\left\lceil P(x)\right\rceil))]:P\in\mathcal P_2\},\\
D_4=\{\#[\neg(\forall xT(\left\lceil P(x)\right\rceil))], \#[\exists xT(\left\lceil P(x)\right\rceil)]:P\in\mathcal P_3\}.
\end{cases}
\end{equation}

Subsets  $G_n(U)$, $n\in\mathbb N_0$, of $D$ are defined 
 recursively as follows.

\begin{equation}\label{E201}
G_0(U)=\begin{cases} W =\{\#A: A \hbox { is a true sentence of $L\}$ if $U=\emptyset$ (the empty set)},\\
    W\cup D_1(U)\cup D_2(U)\cup D_1\cup D_2\cup D_3\cup D_4		\hbox{ if $\emptyset\subset U\subset D$}.
\end{cases}
\end{equation}										

Let  $A$ and $B$ denote sentences of $\mathcal L$. When $n\in\mathbb N_0$, and $G_n(U)$ is defined, define
\begin{equation}\label{E203}
\begin{cases}
G_n^0(U)=\{\#[\neg(\neg A)]:\#A \hbox{ is in } G_n(U)\},\\
G_n^1(U)=\{\#[A\vee B]:\#A \hbox{ or \#$B$ is in } G_n(U)\},\\
G_n^2(U)=\{\#[A\wedge B]:\#A \hbox{ and \#$B$ are in } G_n(U)\},\\
G_n^3(U)=\{\#[A\rightarrow B]:\#[\neg A] \hbox{ or \#$B$ is in } G_n(U)\},\\
G_n^4(U)=\{\#[A\leftrightarrow B]:\hbox{both \#$A$  and \#$B$  or both \#[$\neg A$] and \#[$\neg B$] are in } G_n(U) \},\\
G_n^5(U)=\{\#[\neg(A\vee B)]:\#[\neg A] \hbox{ and \#[$\neg B$] are in } G_n(U)\},\\
G_n^6(U)=\{\#[\neg(A\wedge B)]:\#[\neg A] \hbox{ or \#$[\neg B]$ is in } G_n(U)\},\\
G_n^7(U)=\{\#[\neg(A\rightarrow B)]:\#A \hbox{ and \#[$\neg B$] are in } G_n(U)\},\\
G_n^8(U)=\{\#[\neg(A\leftrightarrow B)]:\hbox{\#$A$ and \#[$\neg B]$, or \#[$\neg A$] and \#$B$ are in } G_n(U) \},
\end{cases}
\end{equation}
and 
\begin{equation}\label{E204}
G_{n+1}(U)=G_n(U)\cup \bigcup_{k=0}^8 G_n^k(U).
\end{equation}
The above constructions imply that
 $G_n(U)\subseteq G_{n+1}(U)\subset D$ and  $G_n^k(U)\subseteq G_{n+1}^k(U)$ for all $n\in\mathbb N_0$ and $k=0,\dots,8$.
Define subsets $G(U)$ and $F(U)$ of $D$ by
\begin{equation}\label{E21}
G(U)=\bigcup_{n=0}^\infty G_n(U), \ \hbox{ and }  F(U)=\{\#A: \#[\neg A]\in G(U)\}.
\end{equation}


\section{A language, its interpretations and their properties}\label{S3}

Let $L$, $\mathcal L$ and $D$  be as in Section \ref{S2}, and let $G(U)$ and $F(U)$, $U\subset D$, be defined by \eqref{E21}. 
Recall that a subset $U$  of $D$ is consistent if there is no sentence $A$ in $\mathcal L$ such that both \#$A$ and \#[$\neg A$] are in $U$. The existence of the smallest consistent subset $U$ of $D$ which satisfies $U=G(U)$ is proved in Appendix by a transfinite recursion method.  
\begin{definition}\label{D1} Let $U$ be the smallest consistent subset of $D$ which satisfies $U=G(U)$. 
 Denote by $\mathcal L^0$ a language which is the set of  all those sentences of $\mathcal L$ whose G\"odel numbers are in $G(U)$ or in $F(U)$.
\end{definition}
\begin{itemize}
\item[(I)]
A theory of syntax for $\mathcal L^0$ is determined by that of the object language $L$, by an extra formula $T(x)$ and its assignments added to $L$ to construct a base language, and by construction of sentences of $\mathcal L^0$ using ordinary and transfinite recursion methods. 
\end{itemize}

\subsection{Definitional interpretation}\label{S31}

An interpretation to sentences of $\mathcal L^0$ is defined as follows.
\begin{itemize}
 \item[(II)] A sentence of $\mathcal L^0$ is interpreted as true iff its G\"odel number is in $G(U)$, and as false 
iff its G\"odel number is in $F(U)$.
\end{itemize}

\begin{lemma}\label{L31} Let the language $\mathcal L^0$ be defined by Definition \ref{D1} and interpreted by (II). Then a sentence of $L$ is true (resp. false) in the interpretation of $L$ iff it is true (resp. false) in the interpretation (II). 
\end{lemma}

\begin{proof} Let $A$ denote a sentence of $L$. $A$ is true in the interpretation (II) iff \#$A$ is in $G(U)$ iff (by the construction of $G(U)$) \#$A$ is in $W$ iff $A$ is true in the interpretation of $L$. 
$A$ is false in the interpretation (II) iff \#$A$ is in $F(U)$ iff 
(by \eqref{E21}) \#[$\neg A$]
is in $G(U)$ iff ($\neg A$ is a sentence of $L$) \#[$\neg A$] is in $W$
iff $\neg A$ is true in the interpretation of $L$ iff (by  negation rule) $A$ is false in the interpretation of $L$. 
\end{proof}

\begin{proposition}\label{P1} 
Let $\mathcal L^0$ be defined by Definition \ref{D1}. Then the interpretation rules given in (ii) of Introduction are satisfied in the  interpretation (II). $T$ is a predicate of $\mathcal L^0$ when its domain is defined by 
\begin{equation}\label{E32}
X_T=\{x: x=\left\lceil A\right\rceil, \ \hbox{ where $A$ is a sentence of} \ \mathcal L^0\}.
\end{equation}
\end{proposition}

\begin{proof} We shall first derive the following auxiliary rule.

(t0) If $A$ is a sentence of $\mathcal L^0$, then $\neg(\neg A)$ is true iff $A$ is true.

To prove (t0), assume first that $\neg(\neg A)$ is true.  Then \#[$\neg(\neg A)$] is in $G(U)$. By \eqref{E21} \#[$\neg(\neg A)$] is in  $G_n(U)$ for some $n\in\mathbb N_0$. If \#[$\neg(\neg A)$] is in $G_0(U)$ then 
 \#[$\neg(\neg A)$] is in $W$, so that sentence $\neg(\neg A)$ is true in the interpretation of $L$. This holds by the negation rule iff $\neg A$ is false in the interpretation of $L$ iff $A$ is true in the interpretation of $L$. Thus \#$A$ is in $W\subset G(U)$, whence $A$ is true.
Assume next that the least of those $n$ for which  \#[$\neg(\neg A)$] is in $G_n(U)$ is $> 0$. The definition of $G_n(U)$ implies that if  \#[$\neg(\neg A)$] is in  $G_n(U)$, then  \#[$\neg(\neg A)$] is in  $G_{n-1}^0(U)$, so that \#$A$ is in $G_{n-1}(U)$, and hence in $G(U)$, i.e., $A$ is true.

Thus $A$ is true if $\neg(\neg A)$ is true.

Conversely, assume that  $A$ is true. Then \#$A$ is in $G(U)$, so that \#$A$  is in $G_n(U)$ for some $n\in\mathbb N_0$. Thus  \#[$\neg(\neg A)$] is in  $G_n^0(U)$, and hence
in $G_{n+1}(U)$. Consequently, \#[$\neg(\neg A)$] is in $G(U)$, whence $\neg(\neg A)$ is true. This concludes the proof of (t0).
\smallskip

Rule (t0) is applied next to prove
\smallskip

\item(t1) Negation rule: $A$ is true iff $\neg A$ is false, and $A$ is false iff $\neg A$ is true.
\smallskip 

Let $A$ be a sentence of $\mathcal L^0$. 
Then  $A$ is true iff (by (t0)) $\neg(\neg A)$ is true iff \#[$\neg(\neg A)$] is in $G(U)$ iff (by \eqref{E21}) 
\#[$\neg A$] is in $F(U)$ iff $\neg A$ is false.

$A$ is false iff \#$A$ is in $F(U)$ iff (by \eqref{E21}) \#[$\neg A$] is in $G(U)$ iff $\neg A$ is true.
Thus (t1) is satisfied.
\smallskip
 
Next we shall prove
\smallskip

\item[(t2)] Disjunction rule: $A\vee B$  is true iff  $A$ or $B$ is true. $A\vee B$ false iff $A$ and $B$ are false. 
 \smallskip

Let $A$ and $B$ be sentences of $\mathcal L^0$. If $A$ or $B$ is true, i.e.,  \#$A$ or \#$B$ is in $G(U)$, there is by \eqref{E21} an $n\in\mathbb N_0$ such that \#$A$ or \#$B$ is in $G_n(U)$. Thus  \#[$A\vee B$] is in $G_n^1(U)$, and hence in $G(U)$, so that $A\vee B$ is true.

Conversely, assume that $A\vee B$ is true, or equivalently, \#[$A\vee B$] is in $G(U)$. Then there is by \eqref{E21} an $n\in\mathbb N_0$ such that \#[$A\vee B$] is in $G_n(U)$.
Assume first that $n=0$. If \#[$A\vee B$] is in $G_0(U)$,  it is in $W$.
Thus $A\vee B$ is true in the interpretation of $L$. Because $L$ is fully interpreted, then $A$ or $B$ is true in the interpretation of $L$, and hence also in the interpretation (II) by Lemma \ref{L31}.

Assume next that the least of those $n$ for which  \#[$A\vee B$] is in $G_n(U)$ is $> 0$. Then  \#[$A\vee B$] is in  $G_{n-1}^1(U)$, so that \#$A$ or \#$B$ is in $G_{n-1}(U)$, and hence in
$G(U)$, i.e., $A$ or $B$ is true.

Consequently, $A\vee B$ is true iff  $A$ or $B$ is true.

It follows from \eqref{E21} that
 \smallskip

(a) \#[$A\vee B$] is  in $F(U)$ iff  \#[$\neg(A\vee B)$] is in $G(U)$.
 \smallskip

If \#[$\neg(A\vee B)$] is in $G(U)$, there is by \eqref{E21} an $n\in\mathbb N_0$ such that \#[$\neg(A\vee B)$] is in $G_{n}(U)$.
Assume that $n=0$. If \#[$\neg(A\vee B)$] is in $G_0(U)$, it is in $W$.
Then  $\neg(A\vee B)$ is true and $A\vee B$ is false in the interpretation of $L$, so that $A$ and $B$ are false and $\neg A$ and $\neg B$ are true in the interpretation of $L$, i.e., \#[$\neg A$] and \#[$\neg B$] are in $W$, and hence in $G(U)$.

Assume next that the least of those $n$ for which  \#[$\neg(A\vee B)$] is in $G_n(U)$ is $> 0$. Then  \#[$\neg(A\vee B)$] is in  $G_{n-1}^5(U)$, so that \#[$\neg A$] and \#[$\neg B$] are in $G_{n-1}(U)$, and hence in
$G(U)$. Thus, \#[$\neg A$] and \#[$\neg B$] are in $G(U)$ if \#[$\neg(A\vee B)$] is in $G(U)$.

Conversely,  if \#[$\neg A$] and \#[$\neg B$] are in $G(U)$, there exist by \eqref{E21} $n_1,n_2\in\mathbb N_0$ such that \#[$\neg A$] is in $G_{n_1}(U)$ and \#[$\neg B$]  is in $G_{n_2}(U)$. Denoting $n=\max\{n_1,n_2\}$, then both \#[$\neg A$] and \#[$\neg B$] are in $G_n(U)$.
This result and the definition of $G_n^5(U)$ imply that \#[$\neg(A\vee B)$] is in $G_n^5(U)$, and hence in $G(U)$. Consequently,
 \smallskip

(b)  \#[$\neg(A\vee B)$] is in $G(U)$ iff \#[$\neg A$] and \#[$\neg B$] are in $G(U)$ iff (by \eqref{E21}) \#$A$ and \#$B$ are in $F(U)$. 
 \smallskip

Thus, by (a) and (b), \#[$A\vee B$] is  in $F(U)$ iff  \#$A$ and \#$B$ are in $F(U)$. But this means that $A\vee B$ is false iff $A$ and $B$ are false.
This concludes the proof of (t2). 
\smallskip

The proofs of the following rules are similar to the above proof of (t2).
\smallskip

(t3) Conjunction rule: $A\wedge B$ is true iff 
$A$ and $B$ are true. $A\wedge B$ is false iff 
$A$ or $B$ is false.

(t4) Implication rule: $A\rightarrow B$ is true iff 
$A$ is false or $B$ is true. $A\rightarrow B$ is  false iff 
$A$ is true and $B$ is false.

(t5) Biconditionality rule:
$A \leftrightarrow B$ is true iff $A$ and $B$ are both true or both false. $A \leftrightarrow B$ is  false iff  $A$ is true and 
$B$ is false or $A$ is false and $B$ is true.
\smallskip
 

Interpretation (II), the fact that $U=G(U)$ and the definition \eqref{E32} of $X_T$ imply that $T(x)$ is for each $x\in X_T$ a sentence, true or false, of $\mathcal L^0$.
Because $U$ is nonempty, then \#[$\exists xT(x)$] is in $G_0(U)$ by \eqref{E20} and \eqref{E201}, and hence in $G(U)$ by \eqref{E21}. Thus $\exists xT(x)$  is by (II) a true sentence of $\mathcal L^0$. If \#$A_1$ is  in $G(U)=U$, then \#$T(x)$ is by \eqref{E20} in $D_1(U)$, and hence in $G(U)$ when $x=\left\lceil A_1\right\rceil$.
This result, \eqref{E32} and (II) imply that $T(x)$ is true  when $x=\left\lceil A_1\right\rceil$, and hence for some $x\in X_T$.
 
Since \#[$\exists xT(x)$] is in $G(U)=U$, and since $U$ is consistent, then 
\#[$\neg\exists xT(x)$] is not in $U=G(U)$. This implies by \eqref{E21} that \#[$\exists xT(x)$] is not in $F(U)$, i.e., by  (II), $\exists xT(x)$ is not false. As shown above, $T(x)$ is  true for $x=\left\lceil A_1\right\rceil$.
Since \#$A_1$ is in $U$ which is consistent, then \#[$\neg A_1$] is  not in $U$. Thus it follows from \eqref{E20} that \#[$\neg T(x)$]  is not in $D_2(U)$, and hence not in $G(U)$ for $x=\left\lceil A_1\right\rceil$. 
The above results mean by \eqref{E21} and (II) that $T(x)$ is not false when $x=\left\lceil A_1\right\rceil$, and hence not false for every $x\in X_T$.

The above proof shows that $T$ satisfies the following rule.
\smallskip

(t6) $\exists xT(x)$ is true iff $T(x)$ is true for some $x\in X_T$. $\exists xT(x)$ is false iff  $T(x)$ is false for every $x\in X_T$.

$\forall xT(x)$ is false, because \#[$\neg\forall xT(x)$] is in $G_0(U)$, and hence in $G(U)$, so that \#[$\forall xT(x)$] is in $F(U)$ by \eqref{E21}. If \#[$\neg A_2$] is in $G(U)=U$, then \#[$\neg T(x)$] is by \eqref{E20} in  $D_2(U)$ 
when $x=\left\lceil A_2\right\rceil$. Thus \#[$\neg T(x)$] is in $G(U)=U$, so that $\neg T(x)$ is true, and hence $T(x)$ is false when $x=\left\lceil A_2\right\rceil$, and hence is false for some $x\in X_T$.

Because \#[$\neg\forall xT(x)$] is in $G(U)=U$ and since $U$ is consistent, then 
\#[$\forall xT(x)$] is not in $U=G(U)$ , whence $\forall xT(x)$ is  not true.

It is shown above that $T(x)$ is false when $x=\left\lceil A_2\right\rceil$. Because \#[$\neg A_2$] is in $U$ and $U$ is consistent, then \#$A_2$ is not in $U$. This implies that 
\#$T(x)$ is not in $D_1(U)$, and hence not in $G(U)=U$, i.e.,   $T(x)$ is not true when $x=\left\lceil A_2\right\rceil$, and hence
not true for every $x\in X_T$. 

The above results imply that $T$ satisfies the following rule.
\smallskip

(t7) $\forall xT(x)$ is true iff $T(x)$ is true for every  $x\in X_T$. $\forall xT(x)$ is  false iff $T(x)$ is false for some $x\in X_T$.
\smallskip

 It remains to show that the following properties hold.
\smallskip

(tt6) $\exists xT(\left\lceil T(x)\right\rceil)$ is true iff $T(\left\lceil T(x)\right\rceil)$ is true for some $x\in X_T$. $\exists xT(\left\lceil T(x)\right\rceil)$ is false iff  $T(\left\lceil T(x)\right\rceil)$ is false for every $x\in X_T$;
\smallskip

(tt7) $\forall xT(\left\lceil T(x)\right\rceil)$ is true iff $T(\left\lceil T(x)\right\rceil)$ is true for every  $x\in X_T$. $\forall xT(\left\lceil T(x)\right\rceil)$ is false iff 
$T(\left\lceil T(x)\right\rceil)$ is false for some $x\in X_T$;
\smallskip

and if $P$ is a predicate of $L$ with domain $X_P$, then 

(tp6) $\exists xT(\left\lceil P(x)\right\rceil)$ is true iff $T(\left\lceil P(x)\right\rceil)$ is true for some $x\in X_P$. $\exists xT(\left\lceil P(x)\right\rceil)$ false iff  $T(\left\lceil P(x)\right\rceil)$ is false for every $x\in X_P$;
\smallskip

(tp7) $\forall xT(\left\lceil P(x)\right\rceil)$ is true iff $T(\left\lceil P(x)\right\rceil)$ is true for every  $x\in X_P$. $\forall xT(\left\lceil P(x)\right\rceil)$ false iff 
$T(\left\lceil P(x)\right\rceil)$ is false for some $x\in X_P$.
\smallskip

To begin with properties (tp6) and (tp7), consider first the case when $P\in\mathcal P_1$. Then $P(x)$ is a true sentence of $L$ for every  $x\in X_P$.
Since $U$ is nonempty, then \#[$\exists xT(\left\lceil P(x)\right\rceil)$] and \#[$\forall xT(\left\lceil P(x)\right\rceil)$] are in $G_0(U)$ by \eqref{E20} and \eqref{E201}, and hence in $G(U)$. Thus $\exists xT(\left\lceil P(x)\right\rceil)$ and $\forall xT(\left\lceil P(x)\right\rceil)$ are by (II)  true sentences of $\mathcal L^0$.  

As shown above \#[$\exists xT(\left\lceil P(x)\right\rceil)$] and \#[$\forall xT(\left\lceil P(x)\right\rceil)$] are in  $G(U)=U$. Because $U$ is consistent, then \#[$\neg(\exists xT(\left\lceil P(x)\right\rceil))$] and 
\#[$\neg(\forall xT(\left\lceil P(x)\right\rceil))$] are not in $U=G(U)$. This implies by \eqref{E21} that \#[$\exists xT(\left\lceil P(x)\right\rceil)$] and \#[$\forall xT(\left\lceil P(x)\right\rceil)$] are  not in $F(U)$, and hence, by  (II), $\exists xT(\left\lceil P(x)\right\rceil)$ and $\forall xT(\left\lceil P(x)\right\rceil)$ are not false. 

Since every true sentence of $L$ is by Lemma \ref{L31} a true sentence of $\mathcal L^0$, then, by the interpretation (II), 
 \#$P(x)$ is in $G(U)=U$ for every $x\in X_P$.
Thus \#$T(\left\lceil P(x)\right\rceil)$ is  by  \eqref{E20} in $D_1(U)$, and hence in $G(U)$, for every $x\in X_P$.
Consequently, by (II),  $T(\left\lceil P(x)\right\rceil)$ is true  for every $x\in X_P$, and hence also for some $x\in X_P$.

As noticed above, \#$P(x)$ is in $G(U)=U$  for every  $x\in X_P$. Since $U$ is consistent, then  \#[$\neg P(x)$]  is not in $U$ for any $x\in X_P$. This result and \eqref{E20} imply that \#[$\neg T(\left\lceil P(x)\right\rceil)]$ is not in $D_2(U)$, and hence not in $G(U)$ for any $x\in X_P$. Thus \#[$T(\left\lceil P(x)\right\rceil)]$ is not in $F(U)$ for any $x\in X_P$.
This means by  (II) that $T(\left\lceil P(x)\right\rceil)$ is not false for any $x\in X_P$.
As shown above,  $T(\left\lceil P(x)\right\rceil)$ is  true for every $x\in X_P$. 

The above results imply that $P$ has properties (tp6) and (tp7) when $P$ is in $\mathcal P_1$. The proof in the case when $P$ is in $\mathcal P_2$ is similar when true is replaced by false and vice versa, and sentences with $\neg$ are replaced by those without $\neg$, and vice versa.

Assume next that $P$ is in $\mathcal P_3$.
 Then $P(x)$ is a true sentence of $L$ for some $x\in X_P$, say $x\in X^1_P$, and a false sentence of $L$ for $x\in X^2_P=X_P\setminus X^1_P$.
Since $U$ is nonempty, then \#[$\exists xT(\left\lceil P(x)\right\rceil)$] and \#[$\neg(\forall xT(\left\lceil P(x)\right\rceil))$] are in $D_4$ by \eqref{E20}, and hence in $G(U)$ by \eqref{E201} and \eqref{E21}. Thus $\exists xT(\left\lceil P(x)\right\rceil)$ and $\neg(\forall xT(\left\lceil P(x)\right\rceil))$ are by (II)  true sentences of $\mathcal L^0$.  

As shown above \#[$\exists xT(\left\lceil P(x)\right\rceil)$] and \#[$\neg(\forall xT(\left\lceil P(x)\right\rceil))$] are in  $G(U)=U$. Because $U$ is consistent, then \#[$\neg(\exists xT(\left\lceil P(x)\right\rceil))$] and 
\#[$\forall xT(\left\lceil P(x)\right\rceil)$] are not in $U=G(U)$. This implies by \eqref{E21} and (t0)  that \#[$\exists xT(\left\lceil P(x)\right\rceil)$] and \#[$\neg(\forall xT(\left\lceil P(x)\right\rceil))$] are  not in $F(U)$, and hence, by  (II), $\exists xT(\left\lceil P(x)\right\rceil)$ and $\neg(\forall xT(\left\lceil P(x)\right\rceil))$ are not false. 

Since every true sentence of $L$ is by Lemma \ref{L31} a true sentence of $\mathcal L^0$ and every false sentence of $L$ is a false sentence of $\mathcal L^0$, then, by the interpretation (II), 
 \#$P(x)$ is in $G(U)=U$ for every $x\in X^1_P$, and in $F(U)$ for every $x\in X^2_P$. In the latter case \#$\neg P(x)$ is in $G(U)$ by \eqref{E21}.
But then, by  \eqref{E20}, \#$T(\left\lceil P(x)\right\rceil)$ is in $D_1(U)$, and hence in $G(U)$, for every $x\in X^1_P$, and \#$\neg T(\left\lceil P(x)\right\rceil)$ is in $D_2(U)$, and hence in $G(U)$, for every $x\in X^2_P$.
Thus, by (II),  $T(\left\lceil P(x)\right\rceil)$ is true  for every $x\in X^1_P$, and $\neg T(\left\lceil P(x)\right\rceil)$ is true for every $x\in X^2_P$.

As noticed above, \#$P(x)$ is in $G(U)=U$  for every  $x\in X^1_P$, and \#[$\neg P(x)$] is in $G(U)=U$ for every $x\in X^2_P$. Since $U$ is consistent, then  \#[$\neg P(x)$]  is not in $U$ for any $x\in X^1_P$ and \#$P(x)$ is not in $U$ for any $x\in X^2_P$. This result and \eqref{E20} imply that \#[$\neg T(\left\lceil P(x)\right\rceil)]$ is not in $D_2(U)$, and hence not in $G(U)$ for any $x\in X^1_P$, and 
\#[$T(\left\lceil P(x)\right\rceil)]$ is not in $D_1(U)$, and hence not in $G(U)$ for any $x\in X^2_P$. Thus \#[$T(\left\lceil P(x)\right\rceil)]$ is not in $F(U)$ for any $x\in X^1_P$ and \#[$\neg T(\left\lceil P(x)\right\rceil)$] is not in $F(U)$ for any $x\in X^2_P$.
This means by  (II) that $T(\left\lceil P(x)\right\rceil)$ is not false for any $x\in X^1_P$, and $\neg T(\left\lceil P(x)\right\rceil)$ is not false for any $x\in X^2_P$.

Consequently, properties (tp6) and (tp7) hold also also  when $P$ is in $\mathcal P_3$.
\smallskip

The proof of properties (tt6) and (tt7) is similar to that given above for properties (tp6) and (tp7) in the case when $P$ is in $\mathcal P_3$.
 
The above results imply that the interpretation rules given in (ii) of Introduction are satisfied in the  interpretation (II) of $\mathcal L^0$.
\end{proof}

We shall derive some properties for the sets $G(U)$ and $F(U)$ defined by \eqref{E21}  when $U$ is a nonempty and consistent subset of $D$. 

\begin{lemma}\label{L01} Let $U$ be nonempty and consistent. Then $G(U)$ and  $F(U)$ are disjoint and consistent.  
\end{lemma}

\begin{proof} If  \#$A$ is in $W$, then $A$ is by \eqref{E201} a true sentence of $L$. Because $L$ is fully interpreted, then $\neg A$ is not true. Thus \#[$\neg A$]  is not in $W$, and hence not in $G(U)$. This implies by \eqref{E21} that \#$A$ is not in  $F(U)$, and hence not in $G_0(U)\cap F(U)$.

Let $x$ be a numeral. If \#$T(x)$ is in $G_0(U)$, it is in $D_1(U)$, so that, by \eqref{E20}, $x=\left\lceil A\right\rceil$, where \#$A$ is in $U$. Because $U$ is consistent, then \#[$\neg A$] is not in $U$. Thus, by \eqref{E20},  
\#[$\neg T(x)]$ is not in $D_2(U)$, and hence not in $G(U)$ when $x=\left\lceil A\right\rceil$. This implies by \eqref{E21} that  \#$T(x)$ is not in $F(U)$ when  $x=\left\lceil A\right\rceil$. Consequently, \#$T(x)$ is not in  $G_0(U)\cap F(U)$.

$U$ is a nonempty, and as a consistent set a proper subset of $D$. Thus \eqref{E201} implies that \#[$\neg\exists xT(x)$] and
\#[$\forall xT(x)$] are in not in $G_0(U)$. By the proof of rules (t6) and (t7) in Proposition \ref{P1} the sentence $\exists xT(x)$ is not false, and the sentence $\forall xT(x)$ is not true, and hence  $\neg\forall xT(x)$ is not false, so that \#[$\exists xT(x)$] and \#[$\neg\forall xT(x)$] are not in $F(U)$. 
Thus none of the  G\"odel numbers \#[$\exists xT(x)$], \#[$\forall xT(x)$], \#[$\neg\exists xT(x)$] and \#[$\neg\forall xT(x)$],
 are in $G_0(U)\cap F(U)$. 

If $P\in\mathcal P$, then by the proof of Proposition \ref{P1} the sentences $\exists xT(\left\lceil P(x)\right\rceil)$,  $\forall xT(\left\lceil P(x)\right\rceil)$, $\forall xT(\left\lceil T(x)\right\rceil)$ and $\neg(\exists xT(\left\lceil T(x)\right\rceil))$ are not both true and false, the sentences $T(\left\lceil P(x)\right\rceil)$  are not both true and false for any $x\in X_P$ and the sentences $T(\left\lceil T(x)\right\rceil)$  are not both true and false for any $x\in X_T$.
Hence the negations of these sentences are not by (t1) both true and false. Thus by (II), their G\"odel numbers are not  in $G(U)\cap F(U)$, so that they are not in $G_0(U)\cap F(U)$. 

The above results and the definition \eqref{E201} of $G_0(U)$ imply that $G_0(U)\cap F(U)=\emptyset$.

Make the induction hypothesis:  
\begin{itemize}
\item[(h0)] 
 $G_n(U)\cap F(U)=\emptyset$ for some $n\in \mathbb N_0$.
\end{itemize}

If \#[$\neg(\neg A)$] would be  in $G_n^0(U)\cap F(U)$, then \#$A$ would be in $G_n(U)$ and \#[$\neg(\neg A)$], or equivalently, by (t1), \#$A$ would be in $F(U)$, so that  \#$A$ would be in $G_n(U)\cap F(U)$. This is impossible by (h0), whence $G_n^0(U)\cap F(U)=\emptyset$.

If \#[$A\vee B$] is in $G_n^1(U)\cap F(U)$,  then \#$A$ or \#$B$ is in $G_n(U)$,  and both \#$A$ and \#$B$ are in $F(U)$ by (t2), so that \#$A$ or 
\#$B$ is in $G_n(U)\cap F(U)$. This contradicts with (h0), whence $G_n^1(U)\cap F(U)=\emptyset$.  

\#[$A\wedge B$] cannot be  in $G_n^2(U)\cap F(U)$, for otherwise both \#$A$ and \#$B$ are in $G_n(U)$, and at least one of \#$A$ and \#$B$ is in $F(U)$, so that \#$A$ or \#$B$ is in $G_n(U)\cap F(U)$, contradicting with (h0). Thus $G_n^2(U)\cap F(U)=\emptyset$.

If \#[$A\rightarrow B$] is in  $G_n^3(U)\cap F(U)$, then \#[$\neg A$] or \#$B$  is in $G_n(U)$ and both \#[$\neg A$] and \#$B$ are in $F(U)$.
But then \#[$\neg A$] or \#$B$  is in $G_n(U)\cap F(U)$. Thus  $G_n^3(U)\cap F(U)=\emptyset$ by (h0).

If \#[$A\leftrightarrow B$] is in $G_n^4(U)\cap F(U)$, then both \#$A$ and \#$B$ or both \#[$\neg A$] and \#[$\neg B$] are in $G_n(U)$,
and both \#[$A$] and \#[$\neg B$] or both \#[$\neg A$] and \#[$B$] are in $F(U)$. Then one of G\"odel numbers \#$A$, \#$B$, \#[$\neg A$] and \#[$\neg B$] is in $G_n(U)\cap F(U)$, contradicting with (h0). Consequently,  $G_n^4(U)\cap F(U)=\emptyset$.

If \#[$\neg(A\vee B)$] is in $G_n^5(U)\cap F(U)$,  then \#[$\neg A$] and \#[$\neg B$] are in $G_n(U)$, and \#[$A\vee B$] is in $G(U)$, i.e.,
\#$A$ or \#$B$ is in $G(U)$, or equivalently, \#[$\neg A$] or \#[$\neg B$] is in $F(U)$. Thus \#[$\neg A$] or \#[$\neg B$] is in $G_n(U)\cap F(U)$. This is impossible by (h0), whence $G_n^5(U)\cap F(U)=\emptyset$.  

If \#[$\neg(A\wedge B)$] is in $G_n^6(U)\cap F(U)$,  then \#[$\neg A$] or \#[$\neg B$] is in $G_n(U)$, and \#[$A\wedge B$] is in $G(U)$, or equivalently, \#$A$ and \#$B$ are in $G(U)$, i.e., \#[$\neg A$] and \#[$\neg B$] are in $F(U)$.  Consequently, \#[$\neg A$] or \#[$\neg B$] is in $G_n(U)\cap F(U)$, contradicting with (h0). Thus $G_n^6(U)\cap F(U)=\emptyset$.  

If \#[$\neg(A\rightarrow B)$] is in $G_n^7(U)\cap F(U)$,  then \#$A$ and \#[$\neg B$] are in $G_n(U)$, and \#[$A\rightarrow B$] is in $G(U)$, i.e.,
\#[$\neg A$] or \#$B$ is in $G(U)$, or equivalently, \#$A$ or \#[$\neg B$] is in $F(U)$. Thus \#$A$ or \#[$\neg B$] is in $G_n(U)\cap F(U)$. This contradicts with (h0), whence $G_n^7(U)\cap F(U)=\emptyset$.  

\#[$\neg(A\leftrightarrow B)$] cannot be  in $G_n^8(U)\cap F(U)$, for otherwise both \#[$A$] and \#[$\neg B$] or both \#[$\neg A$] and \#[$B$] are in $G_n(U)$, and \#[$A\leftrightarrow B$] is in $G(U)$, i.e.,
both \#[$\neg A$] and \#[$\neg B$]  or both \#$A$ and \#$B$ are in $F(U)$. Thus one of G\"odel numbers \#$A$, \#$B$, \#[$\neg A$] and \#[$\neg B$] is in $G_n(U)\cap F(U)$, contradicting with (h0). Thus $G_n^8(U)\cap F(U)=\emptyset$.

Because $G_{n+1}(U)=G_n(U)\cup \bigcup_{k=0}^8 G_n^k(U)$, the above results and the induction hypothesis (h0) imply that $G_{n+1}(U)\cap F(U)=\emptyset$.
Since (h0) is true when $n=0$, it is  by induction true for every $n\in\mathbb N_0$. 

If \#$A$ is in $G(U)$, it is by \eqref{E21} in $G_n(U)$ for some $n\in\mathbb N_0$. Because (h0) is true for every $n\in\mathbb N_0$, then \#$A$ is not in $F(U)$. Consequently, $G(U)\cap F(U)=\emptyset$.

If $G(U)$ is not consistent, then there is such a sentence $A$ of $\mathcal L$, that \#$A$ and \#[$\neg A$] are in $G(U)$. Because \#[$\neg A$] is in $G(U)$, then \#$A$ is also in $F(U)$ by \eqref{E21}, and hence in $G(U)\cap F(U)$. But this is impossible
since  $G(U)\cap F(U)=\emptyset$. Thus $G(U)$ is consistent. 

If $F(U)$ is not consistent, then  \#$A$ and \#[$\neg A$] are in $F(U)$ for some sentence $A$ of $\mathcal L$. Since \#[$\neg A$] is in $F(U)$, then \#$A$ is also in $G(U)$ by \eqref{E21}, and hence in $G(U)\cap F(U)$, a contradiction. 
Thus $F(U)$ is consistent.
\end{proof}

The main result of this subsection is a consequence of Proposition \ref{P1} and Lemma \ref{L01}. 

\begin{proposition}\label{P10} The language $\mathcal L^0$ defined by Definition \ref{D1} is fully interpreted in the interpretation (II).
\end{proposition}

\begin{proof} Let $A$ denote a sentence of $\mathcal L_0$. Then  \#$A$ is in $G(U)$ or in $F(U)$. Because $U$ is consistent, then 
 $G(U)$ and $F(U)$ are disjoint by Lemma \ref{L01}. Thus  \#$A$ is either in $G(U)$ or in $F(U)$, whence $A$ is either true or false in the interpretation (II). The interpretation rules given in (ii) of Introduction are satisfied by Proposition \ref{P1}.
\end{proof} 


\subsection{Semantical interpretation}\label{S32}

A semantical interpretation of a language is given by the following citation from \cite[p. 2]{Ho14}: 
\smallskip

"We say that a language is fully interpreted if all its sentences have meanings that make them either true or false." 
\smallskip

(see also \cite[p. 61]{[9]}). 
This is the case, for instance, when models are used in the interpretation. According to classical logic the interpretation rules (ii) presented in Introduction are assumed to hold also when a language is fully interpreted by meanings of its sentences.

If all sentences of the object language $L$ are equipped with meanings, then
meanings of the sentences of the language $\mathcal L^0$ are determined by meanings of the sentences of $L$, by meaning of the sentence $T(\left\lceil A\right\rceil)$, i.e., 'The sentence denoted by $A$ is true', and by standard meanings of logical symbols.
We shall show that if $L$ is fully interpreted by meanings of its sentences, then $\mathcal L^0$ is fully interpreted by meanings of its sentences, and  this interpretation is equivalent to the interpretation (II).  
In the proof of this result we use the results of the following Lemmas.

 \begin{lemma}\label{l70} Assume that a language $L$ is fully interpreted by  meanings of its sentences, and is without a truth predicate. Let $U$ be as in Definition \ref{D1}, and let $V$ be consistent and satisfy $W\subseteq V\subseteq U$.
If every sentence of $\mathcal L^0$ whose G\"odel number is in $V$ is true and not false by its meaning,
then this property is satisfied when $V$ is replaced by $G(V)$. 
\end{lemma}

\begin{proof} Because $V\subseteq U=G(U)$, then every sentence whose G\"odel number is in $V$, is in $\mathcal L^0$.
 
Given a sentence of $\mathcal L^0$, its  G\"odel number is in $D_1(V)$  iff it is of the form $T(\left\lceil A\right\rceil)$, where $A$ denotes a sentence whose 
G\"odel number is in $V$.  $A$ is by an assumption  true and not false by its meaning.
Thus the sentence $T(\left\lceil A\right\rceil)$, i.e., the sentence 'The sentence denoted by $A$ is true', and hence the given sentence, is true and not false by its meaning. Replacing $A$ by $T(\left\lceil A\right\rceil)$, it follows that $T(\left\lceil T(\left\lceil A\right\rceil)\right\rceil)$ is true and not false by its meaning.

A given sentence of $\mathcal L^0$ has its  G\"odel number in $D_2(V)$ iff it is of the form $\neg T(\left\lceil A\right\rceil)$, where $A$ denotes a sentence of $\mathcal L$, and  the G\"odel number of the sentence $\neg A$ is in $V$.
$\neg A$ is by a hypothesis  true and not false by its meaning, and hence $A$ is false and not true by its meaning. Thus the sentence $T(\left\lceil A\right\rceil)$, i.e., the sentence 'The sentence denoted by $A$ is true', is false and not true by its meaning. 
Replacing $A$ by $T(\left\lceil A\right\rceil)$, we then obtain that $T(\left\lceil T(\left\lceil A\right\rceil)\right\rceil)$ is false and not true by its meaning.
Consequently, by the standard meaning of negation, the sentences $\neg T(\left\lceil A\right\rceil)$ and $\neg T(\left\lceil T(\left\lceil A\right\rceil)\right\rceil)$, and hence also the given sentence, are true and not false by their meanings. 

The domain $X_T$ of $T$, defined by \eqref{E32}, contains $\left\lceil A\right\rceil$ for every  sentence $A$ of $L$.  If $A$ is a true sentence of $L$, then it is  by assumption true and not false by its meaning.
 Since \#$A$ is in $W$, it is also in $V$, whence the sentences $T(\left\lceil A\right\rceil)$ 
and  $T(\left\lceil T(\left\lceil A\right\rceil)\right\rceil)$ are by a result proved above 
true and not false by their meanings. Thus $T(x)$ and $T(\left\lceil T(x)\right\rceil)$ are  for some $x\in X_T$ true and not false  by their meanings. 
These results and the standard meaning of the existential quantifier imply that $\exists x T(x)$ and $\exists xT(\left\lceil T(x)\right\rceil)$ are  true and not false by their meanings.

If $A$ is a false sentence of $L$, then the sentence $\neg A$ is true in the interpretation of $L$, and hence true and not false by its meaning. Since \#[$\neg A$] is in $W$, it is also in $V$, so that by a result proved above the sentence $\neg T(\left\lceil A\right\rceil)$ 
and  $\neg T(\left\lceil T(\left\lceil A\right\rceil)\right\rceil)$ are  true and not false by their meanings. Thus the sentences $\neg T(x)$ and $\neg T(\left\lceil T(x)\right\rceil)$ are  for some $x\in X_T$  true and not false by their meanings, so that $T(x)$ and 
$T(\left\lceil T(x)\right\rceil)$ are for some $x\in X_T$ false and not true by their meanings.  
The above results, the meaning of $T$, and  the standard meanings of the universal quantifier and negation imply that $\neg \forall x T(x)$ and $\neg \forall xT(\left\lceil T(x)\right\rceil)$ are  true and not false by their meanings.
Consequently, the sentences whose G\"odel numbers are in $D_1$ are true and not false by their meanings. 

Let $P$ be a predicate of $L$ with domain $X_P$.  $L$ is by assumption fully interpreted by meaning of its sentences. Thus 
$P(x)$ is for every $x\in X_P$, as a sentence of $L$, either true and not false, or false and not true by its meaning.
This property holds, because of the meaning of $T$, for sentences $T\left\lceil P(x)\right\rceil)$, $x\in X_P$.
By taking also the meanings of the existential and universal quantifiers and negation into account, it follows that    
the sentences whose G\"odel numbers are in  $D_2$, $D_3$ and $D_4$ are true and not false by their meanings.

The above results and the definition \eqref{E201} of $G_0(V)$ imply that every sentence of $\mathcal L^0$ whose G\"odel number is in $G_0(V)$ is true and not false by its meaning. Thus the following property holds when $n=0$.

\begin{itemize}
\item[(h2)] Every sentence of $\mathcal L^0$ whose G\"odel number is in $G_n(V)$ is true and not false by its meaning.
\end{itemize}

Make the induction hypothesis: (h2) holds for some $n\in\mathbb N_0$.

Given a sentence of $\mathcal L^0$ whose  G\"odel number is in $G_n^0(V)$, it is of the form $\neg(\neg A)$, where the G\"odel number of $A$ is in $G_n(V)$. 
By induction hypothesis $A$ is true and not false by its meaning. Thus, by standard meaning of negation, its double application implies that the sentence $\neg(\neg A)$,  and hence the given sentence, is true and not false by its meaning.

A given sentence of $\mathcal L^0$ has its G\"odel number in $G_n^1(V)$ iff it is of the form $A\vee B$, where the G\"odel number of $A$ or $B$ is in $G_n(V)$. 
By the induction hypothesis at least one of the sentences  $A$ and  $B$ is true and not false by its meaning.
 Thus, by the  the standard meaning of disjunction, the sentence
$A\vee B$, and hence  given sentence, is true and not false by its meaning. 

 Similarly it can be shown that if the induction hypothesis holds, then every sentence of $\mathcal L^0$ whose G\"odel number is in  $G_n^k(V)$, where  $2 \le k\le 8$, is true and not false by its meaning. 
 
The above results imply that under the induction hypothesis 
every sentence of $\mathcal L^0$ whose G\"odel number is in $G_n^k(V)$, where  $0 \le k\le 8$, is true and not false by its meaning.

It then follows from the definition \eqref{E204} of $G_{n+1}(V)$ that if (h2) is valid for some $n\in\mathbb N_0$, 
then every sentence of $\mathcal L^0$ whose G\"odel number is in $G_{n+1}(V)$ is true and not false by its meaning.

The first part of this proof shows that (h2) is valid when $n=0$. Thus by induction,
it is valid for all $n\in\mathbb N_0$.
This result and  \eqref{E21} imply that  every sentence of $\mathcal L^0$ whose G\"odel number is in $G(V)$ is true and not false by its meaning.
\end{proof}

\begin{lemma}\label{L9} Assume that a language $L$ is fully interpreted by meanings of its sentences, and has not a truth predicate. Then   the language $\mathcal L^0$ given in Definition \ref{D1} has the following properties.

(a) If a sentence of $\mathcal L^0$ is true in the interpretation (II), it is true and not false by its meaning.

(b) If a sentence of $\mathcal L^0$ is false in the interpretation (II), it is false and not true by its meaning.
\end{lemma}

\begin{proof} By Theorem \ref{T2} the smallest consistent subset $U$ of $D$ which satisfies $U=G(U)$ is the last member of the transfinite sequence $(U_\lambda)_{\lambda<\gamma}$ (indexed by Von Neumann ordinals) constructed in that Theorem.
We prove by transfinite induction that the following result holds for all $\lambda < \gamma$.
\begin{itemize}
\item[(H)] Every sentence  of  $\mathcal L^0$ whose G\"odel number is in $U_\lambda$ is true and not false by its meaning.
\end{itemize}    
Make the induction hypothesis: There exists a $\mu$ which satisfies $0<\mu< \gamma$ such that (H) holds for all $\lambda < \mu$.

Because $U_\lambda$ is consistent and $W\subseteq U_\lambda\subseteq U$ for every $\lambda < \mu$, it follows from the induction hypothesis and Lemma 3.3 that (H) holds  when $U_\lambda$ is replaced by any of the sets $G(U_\lambda)$, $\lambda <\mu$. Thus (H) holds when $U_\lambda$ is replaced by the union of those sets. But this union is $U_\mu$ by Theorem \ref{T2} (C), whence (H) holds when $\lambda =\mu$.

When $\mu =1$, then $\lambda<\mu$ iff $\lambda=0$.
$U_0=W$, i.e., the set of G\"odel numbers of true sentences of $L$. By assumption these sentences are true and not false by their meanings. 
Moreover, these sentences are also sentences of $\mathcal L^0$, since it contains sentences of $L$. 
This proves that the induction hypothesis is satisfied  when $\mu=1$.   
 
The above proof implies by transfinite induction, that properties assumed in (H) for $U_\lambda$ are valid whenever 
 $\lambda <\gamma$. In particular, they are valid for the last member of   $(U_\lambda)_{\lambda<\gamma}$,
which is by Theorem \ref{T2} the smallest consistent subset $U$  of $D$ for which $U=G(U)$. 
Thus every sentence of $\mathcal L^0$, which is true in the interpretation (II), has its G\"odel number in $U$, and hence by is the above proof true and not false by its meaning. This  proves (a).

To prove (b), let $A$ denote a sentence which is false in the interpretation (II). Negation rule implies that $\neg A$ is true in the interpretation (II). Thus, by (a), $\neg A$ is true and not false by its meaning, so that by the standard meaning of negation, $A$ is false and not true by its meaning.
This proves (b).
\end{proof}

Now we are ready to prove the main result of this subsection. 
        
\begin{proposition}\label{P7} Assume that a language  $L$ is fully interpreted by meanings of its sentences, and is without a truth predicate. Then the language $\mathcal L^0$ defined in Definition \ref{D1} is fully interpreted by meanings of its sentences, and this interpretation is equivalent to the interpretation (II). 
\end{proposition}

\begin{proof}
Let $A$ denote a sentence of $\mathcal L^0$. $A$ is by Proposition \ref{P10} either true or false 
in the interpretation (II). If $A$ is true in the interpretation (II), it is by Lemma \ref{L9} (a) true and not false by its meaning.
If $A$ is false in the interpretation (II), it is by Lemma \ref{L9} (b) false and not true by its meaning. Consequently, $A$ is either true  or false by its meaning.  Thus every sentence of $\mathcal L^0$ is either true or false by its meaning, 
or equivalently, $\mathcal L^0$ is fully interpreted by meanings of its sentences,
 and this interpretation is by the above proof equivalent to the interpretation (II). 
\end{proof}


\section{Theory DSTT and its properties}\label{S4}

The next theorem provides for the language $\mathcal L^0$ defined in Definition \ref{D1} a theory of truth. Because the interpretation of $\mathcal L^0$ can be definitional or semantical, we call that theory definitional/semantical theory of truth, shortly DSTT. 

\begin{theorem}\label{T1}  Assume that an object language $L$ is fully interpreted, and is without a truth predicate. Then the language $\mathcal L^0$ defined by Definition \ref{D1} is fully interpreted by (II), and also by meanings of its sequences if $L$ is so interpreted. Moreover, $A\leftrightarrow T(\left\lceil A\right\rceil)$ is true and $A\leftrightarrow \neg T(\left\lceil A\right\rceil)$ is false for every sentence $A$ of $\mathcal L^0$, and $T$ is a truth predicate for $\mathcal L^0$.
\end{theorem}

\begin{proof} $\mathcal L^0$ is by Proposition \ref{P10} fully interpreted in the interpretation (II). If $L$ is fully interpreted by meanings of its sequences, so is $\mathcal L^0$ by Proposition \ref{P7}, and this interpretation is equivalent to the interpretation (II).

Let  $A$ denote a sentence of $\mathcal L^0$.
The interpretation (II), rule (t1), the definitions of $D_1(U)$, $D_2(U)$ and $G(U)$, and the assumption  $U=G(U)$ in Definition \ref{D1}
imply that
\newline
-- $A$ is  true iff \#$A$ is in $G(U)$ iff \#$A$ is in $U$  iff 
\#$T(\left\lceil A\right\rceil)$ is in $G(U)$ and \#[$\neg T(\left\lceil A\right\rceil)$] is in $F(U)$ iff $T(\left\lceil A\right\rceil)$ is true and  $\neg T(\left\lceil A\right\rceil)$ is false;
\newline
-- $A$ is false iff \#$A$ is in $F(U)$ iff \#[$\neg A$] is in $G(U)$ iff \#[$\neg A$] is in $U$  iff \#[$\neg T(\left\lceil A\right\rceil)$] is in $G(U)$  and \#$T(\left\lceil A\right\rceil)$ is in $F(U)$
 iff $\neg T(\left\lceil A\right\rceil)$ is true and $T(\left\lceil A\right\rceil)$ is false.  
\newline
The above results and rule (t5) imply that $A\leftrightarrow T(\left\lceil A\right\rceil)$ is true and $A\leftrightarrow \neg T(\left\lceil A\right\rceil)$ is false for every sentence $A$ of $\mathcal L^0$.

 It follows from Proposition \ref{P1} that $T$ is a predicate of $\mathcal L^0$. Moreover, \eqref{E32} implies that the domain $X_T$ of $T$ is the set of numerals of G\"odel numbers of
all sentences of $\mathcal L^0$. Thus $X_T$ satisfies the following condition presented in \cite[p. 7]{Fe}: "In the case of the truth
predicate $T$, the domain ... is taken to consist of the sentences that
are meaningful and determinate, i.e. have a definite truth value, true or false." (In \cite{Fe}  numerals of G\"odel numbers of
sentences are replaced by sentences itself.)
Consequently, $T$ is a truth predicate for $\mathcal L^0$.
\end{proof}


The next two Lemmas deal with compositionality properties of $T$.

\begin{lemma}\label{L41} Let $\mathcal L^0$ and its interpretation be as in Theorem \ref{T1}. Then the following biconditionals are true  for all sentences $A$ and $B$ of $\mathcal L^0$.

(a1)  $T(\left\lceil \neg A\right\rceil)\leftrightarrow \neg T(\left\lceil A\right\rceil)$.

(a2) $T(\left\lceil A\vee B\right\rceil) \leftrightarrow  T(\left\lceil A\right\rceil)\vee T(\left\lceil B\right\rceil)$. 

(a3) $T(\left\lceil A\wedge B\right\rceil) \leftrightarrow  T(\left\lceil A\right\rceil)\wedge T(\left\lceil B\right\rceil)$. 

(a4) $T(\left\lceil A\rightarrow B\right\rceil) \leftrightarrow  (T(\left\lceil A\right\rceil)\rightarrow T(\left\lceil B\right\rceil))$. 

(a5) $T(\left\lceil A\leftrightarrow B\right\rceil) \leftrightarrow  (T(\left\lceil A\right\rceil)\leftrightarrow T(\left\lceil B\right\rceil))$. 
\end{lemma}

\begin{proof}
Let $A$ be a sentence of $\mathcal L^0$. Then

$T(\left\lceil \neg A\right\rceil)$ is true iff $\neg A$ is true iff $A$ is false iff $T(\left\lceil A\right\rceil)$ is false iff
 $\neg T(\left\lceil A\right\rceil)$ is true.

$T(\left\lceil \neg A\right\rceil)$ is false iff $\neg A$ is false iff $A$ is true iff $T(\left\lceil A\right\rceil)$ is true iff
 $\neg T(\left\lceil A\right\rceil)$ is false.

Thus (a1) is true.

Let $A$ and $B$ be sentences of $\mathcal L^0$. Then

$T(\left\lceil A\vee B\right\rceil)$ is true iff $A\vee B$ is true iff $A$ or $B$ is true iff $T(\left\lceil A\right\rceil)$ or $T(\left\lceil B\right\rceil)$ is true iff
 $T(\left\lceil A\right\rceil)\vee T(\left\lceil B\right\rceil)$ is true.

$T(\left\lceil A\vee B\right\rceil)$ is false iff $A\vee B$ is false iff $A$ and $B$ are false iff $T(\left\lceil A\right\rceil)$ and $T(\left\lceil B\right\rceil)$ are false iff
 $T(\left\lceil A\right\rceil)\vee T(\left\lceil B\right\rceil)$ is false.

Consequently, (a2) is true. The proof that (a3) and (a4) are true is similar.

Let $A$ and $B$ be sentences of $\mathcal L^0$. Then

$T(\left\lceil A\leftrightarrow  B\right\rceil)$ is true iff $A\leftrightarrow  B$ is true iff $A$ and $B$ are both true or both false iff $T(\left\lceil A\right\rceil)$ and $T(\left\lceil B\right\rceil)$  are both true or both  false  iff
 $T(\left\lceil A\right\rceil)\leftrightarrow  T(\left\lceil B\right\rceil)$ is true.

$T(\left\lceil A\leftrightarrow  B\right\rceil)$ is false iff $A\leftrightarrow  B$ is false iff $A$ is true and $B$ is false or $A$ is false and $B$ is true iff $T(\left\lceil A\right\rceil)$ is true and $T(\left\lceil B\right\rceil)$ is false or $T(\left\lceil A\right\rceil)$ is false and $T(\left\lceil B\right\rceil)$ is true iff
 $T(\left\lceil A\right\rceil)\leftrightarrow T(\left\lceil B\right\rceil)$ is false.
Consequently, (a5) is true.
\end{proof}

\begin{lemma}\label{L42} Let $\mathcal L^0$ and its interpretation be as in Theorem \ref{T1}.
If $P$ is a predicate of $L$ with domain $X_P$, then the following biconditionals are true.

(a6) $T(\left\lceil \forall xP(x)\right\rceil) \leftrightarrow  \forall xT(\left\lceil P(x)\right\rceil)$. 

(a7) $T(\left\lceil \exists xP(x)\right\rceil) \leftrightarrow  \exists xT(\left\lceil P(x)\right\rceil)$. 

(a8) $T(\left\lceil \forall xT(x)\right\rceil) \leftrightarrow  \forall xT(\left\lceil T(x)\right\rceil)$. 

(a9) $T(\left\lceil \exists xT(x)\right\rceil) \leftrightarrow  \exists xT(\left\lceil T(x)\right\rceil)$. 
\end{lemma}

\begin{proof} 
 $T(\left\lceil \forall xP(x)\right\rceil)$ is true iff $\forall xP(x)$ is true iff $P(x)$
is true for every $x\in X_P$ iff $T(\left\lceil P(x)\right\rceil)$ is true for every $x\in X_P$,  iff (by (tp7)) $\forall xT(\left\lceil P(x)\right\rceil)$ is true.
Consequently,

(a61) $T(\left\lceil \forall xP(x)\right\rceil)$ is true  iff $\forall xT(\left\lceil P(x)\right\rceil)$ is true. 

Assume that $T(\left\lceil \forall xP(x)\right\rceil)$ is false. If $\forall xT(\left\lceil P(x)\right\rceil)$ would be true, then 
 $T(\left\lceil \forall xP(x)\right\rceil)$ would be true by (a61). But then  $T(\left\lceil \forall xP(x)\right\rceil)$ would be both false and true, which is impossible because $\mathcal L^0$ is fully interpreted. Thus $\forall xT(\left\lceil P(x)\right\rceil)$ is false.
Similarly it can be shown that if $\forall xT(\left\lceil P(x)\right\rceil)$ is false, then $T(\left\lceil \forall xP(x)\right\rceil)$ is false.
The above results imply that

(a62) $T(\left\lceil \forall xP(x)\right\rceil)$ is false  iff $\forall xT(\left\lceil P(x)\right\rceil)$ is false. 

Results (a61), (a62) and (t5) imply that (a6) is true.

$T(\left\lceil \exists xP(x)\right\rceil)$ is true iff $\exists xP(x)$ is true iff $P(x)$
is true for some $x\in X_P$ iff $T(\left\lceil P(x)\right\rceil)$ is true for some $x\in X_P$, iff (by (tp6)) $\exists xT(\left\lceil P(x)\right\rceil)$ is true.
Consequently,

(a71) $T(\left\lceil \exists xP(x)\right\rceil)$ is true  iff $\exists xT(\left\lceil P(x)\right\rceil)$ is true. 

Result (a71) and the fact that $\mathcal L^0$ is fully interpreted imply (cf. the proof of (a62)) that

(a72) $T(\left\lceil \exists xP(x)\right\rceil)$ is false  iff $\exists xT(\left\lceil P(x)\right\rceil)$ is false.

As a consequence of (a71), (a72) and (t5) we obtain that (a7) is true. 

Truths of (a8) and (a9) are proved  similarly as truths of (a6) and (a7) when properties (tp6) and (tp7) are replaced by (tt6) and (tt7).
\end{proof}






Hannes Leitgeb formulated in his paper \cite{Lei07}:  'What Theories of Truth Should be Like (but Cannot be)' the following norms for theories of truth:
\begin{itemize}
\item[(n1)] Truth should be expressed by a predicate (and a theory of syntax should be available).

\item[(n2)]  If a theory of truth is added to mathematical or empirical theories, it
should be possible to prove the latter true.

\item[(n3)]  The truth predicate should not be subject to any type restrictions.

 \item[(n4)] $T$-biconditionals should be derivable unrestrictedly.

 \item[(n5)]  Truth should be compositional.

\item[(n6)]  The theory should allow for standard interpretations.

 \item[(n7)]  The outer logic and the inner logic should coincide.

\item[(n8)]  The outer logic should be classical.
\end{itemize}

The next Theorem shows that theory DSTT satisfies these norms, and also two additional norms. 
The author considers validity the extra norm (n9) as crucial to every theory of truth.   

\begin{theorem}\label{T31} The theory of truth DSTT formulated in Theorem \ref{T1} 
satisfies the norms (n1)--(n8) and 
the following norms.
\smallskip

(n9) The theory of truth should be free from paradoxes.
\smallskip

(n10) Truth should be
explained for the language in which this very theory is expressed.
\end{theorem}

\begin{proof} (n1): $T$ is by Theorem \ref{T1} a truth predicate for the language $\mathcal L^0$ of theory DSTT. (A theory of syntax is available, as stated in (I) after Definition \ref{D1}).
\smallskip

(n2): By Lemma \ref{L31} DSTT proves theories of the object language $L$ true. 
\smallskip

(n3): $T$ is by Theorem \ref{T1} a truth predicate for the language $\mathcal L^0$ of theory DSTT, and is not subject to any restrictions in 
that language.
\smallskip

(n4):  By Theorem \ref{T1}  $T$-biconditionals $A\leftrightarrow T(\left\lceil A\right\rceil)$ are derivable unrestrictedly in $\mathcal L^0$.
\smallskip

(n5):  Lemmas \ref{L41} and \ref{L42} imply that truth in DSTT is compositional.
\smallskip

(n6): DSTT allows for standard interpretations. For instance, the interpretation of the language $\mathcal L^0$ of theory DSTT by meanings of its sentences is possible if the object language is so interpreted. Examples of such interpretations are given in Introduction. 
\smallskip

(n7) and (n8): The interpretation rules given in (ii) of Introduction, assumed
for $L$, and proved for $\mathcal L^0$ in Proposition \ref{P1}, are those of classical logic. Consequently, 
both  the outer logic and the inner logic are classical in DSTT.  
\smallskip

(n9): Every sentence of $\mathcal L^0$ is either true or false in both interpretations. Thus DSTT is free from paradoxes. 
\smallskip

(n10): In (II) truth is explained (an interpretation is explained in English) for the language  $\mathcal L^0$ where theory DSTT is expressed. 
In the semantical case the interpretation is explained by meanings of the sentences of  $\mathcal L^0$.
\end{proof}

\section{Remarks}\label{S7}


Compared with \cite{Hei14,Hei15,Hei16}, the language for the presented theory of truth contains new sentences, the domain of $T$ is so chosen that it meets the requirements presented in \cite{Fe}, and the collection of possible object languages $L$ is larger. Neither compositionality results for $T$ nor the interpretation of $\mathcal L^0$ by meanings of its sentences, and the equivalence if this interpretation to the interpretation (II) are presented in those papers. Also some proofs are simplified and specified. 
\smallskip

As in \cite{Hei16}, the sentences $\exists x [\neg T(x)]$ and $\forall x [\neg T(x)]$ could be added to $\mathcal L^0$ so that the biconditionals
$\forall x [\neg T(x)]\leftrightarrow [\neg(\exists x T(x))]$ and $\exists x [\neg T(x)]\leftrightarrow [\neg(\forall x T(x))]$ would be true. Thus also $\neg T$  would be a predicate of $\mathcal L^0$. 
\smallskip

If $U$ is a consistent subset of $D$ which satisfies $U=G(U)$, but not the smallest, and if $L$ is fully interpreted, then $T$ is a truth predicate for the language $\mathcal L^0$, defined as in Definition \ref{D1}, and interpreted by (II), using that $U$. The so obtained theory of truth satisfies norms (n0)--(n10). It is questionable whether there exists a proof that for such $U$ the language $\mathcal L^0$ is fully interpreted by meanings of its sentences when  $L$ is so interpreted. 
\smallskip

The set $U$ in theory DSTT is the smallest consistent subset of $D$ which satisfies $U=G(U)$.
Thus the sentences of $\mathcal L^0$ are grounded in the sense defined by Kripke in \cite[p. 18]{[15]}. 
\smallskip

In Tarski's theory of truth (cf. \cite{[20]}) the truth predicate is in every step for the  language preceding that step. The sentences of that language do not contain the truth predicate in question. Thus the norms (n3) and (n4) are not satisfied.  
\smallskip

Leitgeb gives in \cite[p. 9]{Lei07} the following justification to his opinion that there cannot be a theory of truth which satisfies all the norms presented in \cite{Lei07}: "Consider a first-order theory which conforms to these norms, such that truth is to be explained for the language in which this very theory is expressed. From the theory of syntax the existence of a so-called Liar sentence is derivable." 
This means that
 the existence of a sentence $A$ for which $A\leftrightarrow \neg T(\left\lceil A\right\rceil)$ "follows from the syntactic axioms", and hence is true. But also $A\leftrightarrow T(\left\lceil A\right\rceil)$
is true by (n4). Thus $A$ is both true and false by the biconditionality rule (t5) of classical logic, so that 
such a theory of truth is contradictory, whereas DSTT is not.  While the theory of syntax for the object language  $L$ can be that of a first-order theory, the theory of syntax for the language $\mathcal L^0$ of theory DSTT is not.
\smallskip

In \cite[Conclusions]{Fe} Feferman urges  "the pursuit of axiomatizations of semantical or definitional approaches that have not yet been
thus treated, and the close examination of them in the light of the given criteria." By 'the given criteria' Feferman means Leitgeb's norms (n1)--(n8). The close examination of theory DSTT in the light of those criteria is carried out above. However, neither axioms for DSTT nor an axiomatic theory of truth conforming to norms (n1)--(n10) can be constructed. Reasons for this are found from (\cite{Fi15}).


\section{Appendix}\label{S8}
Before the proof of Theorem \ref{T2} we shall first prove auxiliary results, using the concepts adopted in previous sections.

\begin{lemma}\label{L203} Assume that $U$ and $V$ are consistent  subsets of  $D$, and that $V\subseteq U$.
 Then $G(V)\subseteq G(U)$ and $F(V)\subseteq F(U)$.  
\end{lemma}

\begin{proof} As consistent sets both $V$ and $U$ are proper subsets of $D$.
 
Let $A$ be a sentence of $L$. Definition of $G(U)$ implies that \#$A$ is in $G(U)$ and also in $G(V)$ iff \#$A$ is in $W$.
  
If \#$T(x)$ is in $D_1(V)$,  then   $x=\left\lceil A\right\rceil$, where \#$A$ is in $V$. Because $V\subseteq U$, then \#$A$ is also in $U$, whence \#$T(x)$ is in $D_1(U)$.

If \#[$\neg T(x)]$ is in $D_2(V)$,  then  $x$ is $\left\lceil A\right\rceil$, where \#[$\neg A$] is in $V$. Because $V\subseteq U$, then \#[$\neg A$] is also in $U$, whence \#[$\neg T(x)$] is in $D_2(U)$.

If \#[$\exists xT(x)$] is  in $G_0(V)$, then $V$ is nonempty. Because $V\subseteq U$, then also $U$ is nonempty, whence \#[$\exists xT(x)]$ is in $G_0(U)$. Consequently,   
\#[$\exists xT(x)$] is in $G_0(U)$ whenever it is in $G_0(V)$. The similar reasoning shows that \#[$\neg\forall x T(x))$] is in $G_0(U)$ whenever it is in $G_0(V)$. The sets $D_1$, $D_2$, $D_3$ and $D_4$ are contained in $U$ if they are contained in $V$.

The above results imply that $G_0(V)\subseteq G_0(U)$. 
Make an induction hypothesis:
\begin{enumerate}
\item[(h1)] $G_n(V)\subseteq G_n(U)$.
\end{enumerate}
The definitions of the sets $G_n^k(U)$, $k=0,\dots,8$, given in \eqref{E203}, together with (h1), imply that
$G_n^k(V)\subseteq G_n^k(U)$ for each $k=0,\dots,8$. Thus
$$
G_{n+1}(V)=G_n(V)\cup \bigcup_{k=0}^8 G_n^k(V)\subseteq G_n(U)\cup \bigcup_{k=0}^8 G_n^k(U)=G_{n+1}(U).
$$
Because (h1) is true when $n=0$, then it is true for every $n\in\mathbb N_0$.

If \#$A$ is in $G(V)$, it is by \eqref{E21} in $G_n(V)$ for some $n$. Thus   \#$A$ is in $G_n(U)$ by (h1), and hence in $G(U)$.
Consequently, $G(V)\subseteq G(U)$. 

If \#$A$ is in $F(V)$, it follows from \eqref{E21} that \#[$\neg A$] is in $G(V)$. Because $G(V)\subseteq G(U)$, then \#[$\neg A$] is in $G(U)$.
This implies  by \eqref{E21} that  \#$A$ is in $F(U)$. Thus $F(V)\subseteq F(U)$.
\end{proof}

Denote by $\mathcal C$ the family of consistent subsets of $D$.
In the formulation and the proof of Theorem \ref{T2} below  transfinite sequences indexed by von Neumann ordinals are used. A transfinite sequence $(U_\lambda)_{\lambda<\alpha}$ of $\mathcal C$ is said to be increasing if
$U_\mu\subseteq U_\nu$ whenever $\mu<\nu<\alpha$, and strictly increasing if
$U_\mu\subset U_\nu$ whenever $\mu<\nu<\alpha$.

\begin{lemma}\label{L204} Assume that $(U_\lambda)_{\lambda<\alpha}$ a strictly increasing sequence of $\mathcal C$.
Then

(a) $(G(U_\lambda))_{\lambda<\alpha}$  is an increasing sequence of $\mathcal C$.

(b) The set $U_\alpha = \underset{\lambda<\alpha}{\bigcup}G(U_\lambda)$ is consistent.
\end{lemma}

\begin{proof} (a) Consistency of the sets $G(U_\lambda)$, $\lambda<\alpha$, follows from Lemma \ref{L01} because the sets  $U_\lambda$, $\lambda<\alpha$, are consistent.

Because $U_\mu\subset U_\nu$ whenever $\mu<\nu<\alpha$, then $G(U_\mu)\subseteq G(U_\nu)$ whenever $\mu<\nu<\alpha$, by Lemma \ref{L203}, whence the sequence  $(G(U_\lambda))_{\lambda<\alpha}$  is increasing. This proves (a).
\smallskip

To prove that the set $\underset{\lambda<\alpha}{\bigcup}G(U_\lambda)$  is consistent, assume on the contrary that there exists such a sentence $A$ in $\mathcal L$ that  both \#$A$ and \#[$\neg A$] are in $\underset{\lambda<\alpha}{\bigcup}G(U_\lambda)$.
Thus there exist $\mu,\,\nu<\alpha$ such that \#$A$ is in $G(U_\mu)$ and \#[$\neg A$] is in $G(U_\nu)$. Because
$G(U_\mu)\subseteq G(U_\nu)$ or $G(U_\nu)\subseteq G(U_\mu)$, then both  \#$A$ and \#[$\neg A$] are in $G(U_\mu)$ or in $G(U_\nu)$.
But this is impossible, since both $G(U_\mu)$ and $G(U_\nu)$ are consistent by (a). Thus, the set $\underset{\lambda<\alpha}{\bigcup}G(U_\lambda)$  is consistent. This proves the conclusion of (b).
\end{proof}


\begin{theorem}\label{T2} 
 The union of those transfinite sequences 
$(U_\lambda)_{\lambda<\alpha}$ of $\mathcal C$ which satisfy
\begin{itemize}
\item[(C)] $(U_\lambda)_{\lambda<\alpha}$ is strictly increasing,  
$U_0=W$, and if $0<\mu< \alpha$, then
$U_\mu = \underset{\lambda<\mu}{\bigcup}G(U_\lambda)$
\end{itemize}  
is a transfinite sequence. It has the last member, which is the smallest consistent subset $U$ of $D$  which satisfies $U=G(U)$. 
\end{theorem}

\begin{proof} 
Those transfinite sequences of $\mathcal C$ which satisfy condition (C)
are called  $G$-sequences. 
We shall first show that $G$-sequences are nested:
\begin{enumerate}
\item[(1)] {\it Assume that $(U_\lambda)_{\lambda<\alpha}$ and $(V_\lambda)_{\lambda<\beta}$ are $G$-sequences, and that $\{U_\lambda\}_{\lambda<\alpha} \not\subseteq \{V_\lambda\}_{\lambda<\beta}$.
Then $(V_\lambda)_{\lambda<\beta}  = (U_\lambda)_{\lambda<\beta}$}.
\end{enumerate}
By the assumption of (1)  $\mu = \min \{\lambda<\alpha\mid U_\lambda\not\in\{V_\lambda\}_{\lambda<\beta}\}$ exists, and 
$\{U_\lambda\}_{\lambda<\mu} \subseteq \{V_\lambda\}_{\lambda<\beta}$.  
Properties (C) imply by transfinite induction that $U_\lambda=V_\lambda$ for each $\lambda<\mu$. To prove that $\mu=\beta$,
 make a counter-hypothesis: $\mu<\beta$. Since $\mu<\alpha$ and $U_\lambda=V_\lambda$ for each $\lambda<\mu$, it follows from properties (C) that   $U_\mu = \underset{\lambda<\mu}{\bigcup}G(U_\lambda) = \underset{\lambda<\mu}{\bigcup}G(V_\lambda)=V_\mu$, which is
impossible, since $V_\mu\in \{V_\lambda\}_{\lambda<\beta}$, but $U_\mu\not\in \{V_\lambda\}_{\lambda<\beta}$. Consequently, $\mu=\beta$ and $U_\lambda=V_\lambda$ for each $\lambda<\beta$, whence
 $(V_\lambda)_{\lambda<\beta} =(U_\lambda)_{\lambda<\beta}$.

By definition, every $G$-sequence $(U_\lambda)_{\lambda<\alpha}$ is a function $\lambda\mapsto U_\lambda$ from $\alpha$ into $\mathcal C$. Property (1) implies that these functions are compatible. Thus their union is by \cite[Theorem 2.3.12]{[12]} a function with values in $\mathcal C$, the domain being the union of all index sets  of $G$-sequences. Because these index sets are ordinals, then their union is also an ordinal by \cite[I.8.10]{[Ku]}. Denote it by $\gamma$. The union function can be represented as a sequence $(U_\lambda)_{\lambda<\gamma}$ of $\mathcal C$. It is strictly increasing as a union of strictly increasing nested sequences.  

To show that $\gamma$ is a successor, assume on the contrary that $\gamma$ is a limit ordinal. Given $\nu<\gamma$, then $\mu=\nu+1$ and $\alpha=\mu+1$ are in $\gamma$, and $(U_\lambda)_{\lambda<\alpha}$ is a $G$-sequence. Denote $U_\gamma = \underset{\lambda<\gamma}{\bigcup}G(U_\lambda)$.
 $G$ is order preserving  by Lemma \ref{L203}, and $(U_\lambda)_{\lambda<\gamma}$ is a strictly increasing sequence of $\mathcal C$. Thus  $(G(U_\lambda))_{\lambda<\gamma}$ is increasing by  Lemma \ref{L204}(a), and $U_\gamma$ is consistent by Lemma \ref{L204}(b). Moreover, 
 $U_\nu\subset U_\mu=\underset{\lambda<\mu}{\bigcup}G(U_\lambda)\subseteq U_\gamma$. This is true for each $\nu<\gamma$, whence $(U_\lambda)_{\lambda<\gamma+1}$ is a $G$-sequence. This is impossible, since $(U_\lambda)_{\lambda<\gamma}$ is the union of all $G$-sequences.
Consequently, $\gamma$ is a successor, say $\gamma=\alpha+1$. Thus $U_\alpha$ is the last member of $(U_\lambda)_{\lambda<\gamma}$,
$U_\alpha=\max\{U_\lambda\}_{\lambda<\gamma}$, and $G(U_\alpha)=\max\{G(U_\lambda)\}_{\lambda<\gamma}$. Moreover, $(U_\lambda)_{\lambda<\gamma}$ is a $G$-sequence, for otherwise $(U_\lambda)_{\lambda<\alpha}$ would be the union of all $G$-sequences. In particular,
$U_\alpha=\underset{\lambda<\alpha}{\bigcup}G(U_\lambda)\subseteq \underset{\lambda<\gamma}{\bigcup}G(U_\lambda)=G(U_\alpha)$,
so that $U_\alpha\subseteq G(U_\alpha)$. This inclusion cannot be proper, since then the longest $G$-sequence $(U_\lambda)_{\lambda<\gamma}$ could be extended by  $U_\gamma= \underset{\lambda<\gamma}{\bigcup}G(U_\lambda)$. Consequently, $U_\alpha = G(U_\alpha)$.

Assume that $U$ is a consistent subset of $D$, and that $U=G(U)$. Then $U_0=W=G(\emptyset)\subseteq G(U)=U$. If $0<\mu<\gamma$, and $U_\lambda\subseteq U$ for each $\lambda<\mu$, then $G(U_\lambda)\subseteq G(U)$ for each $\lambda<\mu$, whence $U_\mu=\underset{\lambda<\mu}{\bigcup}G(U_\lambda)\subseteq G(U)=U$. Thus, by transfinite induction, $U_\mu\subseteq U$ for each $\mu<\gamma$. This proves the last assertion of Theorem.
\end{proof}

{}

\end{document}